\documentclass[12pt,a4paper,reqno]{amsart}
\usepackage{amsmath,amssymb,amsthm,wasysym,calc,verbatim,enumitem,tikz,url,hyperref,mathrsfs,bbm,cite,fullpage}
\usetikzlibrary{shapes.misc,calc,intersections,patterns,decorations.pathreplacing}
\hypersetup{colorlinks=true,linktoc=page}
\hypersetup{%
    colorlinks,
    linkcolor={red!50!black},
    citecolor={green!50!black},
    urlcolor={blue!50!black}
}

\usepackage[abbrev,msc-links,backrefs]{amsrefs}
\usepackage{doi}

\renewcommand{\PrintDOI}[1]{\doi{#1}}

\AtBeginDocument{\def\MR#1{}}

\usepackage{tikz}
\usetikzlibrary{decorations.markings}
\usetikzlibrary{calc,positioning,decorations.pathmorphing,decorations.pathreplacing}

\usepackage{fullpage}
\usepackage{setspace}
\usepackage{datetime}
\usepackage[margin=1cm]{caption}

\newtheorem{theorem}{Theorem}
\newtheorem{lemma}[theorem]{Lemma}
\newtheorem{proposition}[theorem]{Proposition}

\newtheorem{corollary}[theorem]{Corollary}
\newtheorem{conj}[theorem]{Conjecture}

\newtheorem{problem}[theorem]{Problem}

\theoremstyle{definition}
\newtheorem{definition}[theorem]{Definition}
\newtheorem{remark}[theorem]{Remark}

\usepackage{enumitem}

\let\polishlcross=\l
\def\l{\ifmmode\ell\else\polishlcross\fi}

\makeatletter
\def\moverlay{\mathpalette\mov@rlay}
\def\mov@rlay#1#2{\leavevmode\vtop{    \baselineskip\z@skip\lineskiplimit-\maxdimen%
    \ialign{\hfil$\m@th#1##$\hfil\cr#2\crcr}}}
\newcommand{\charfusion}[3][\mathord]{
    #1{\ifx#1\mathop\vphantom{#2}\fi
        \mathpalette\mov@rlay{#2\cr#3}
      }
    \ifx#1\mathop\expandafter\displaylimits\fi}
\makeatother

\DeclareFontFamily{U}  {MnSymbolC}{}
\DeclareSymbolFont{MnSyC}         {U}  {MnSymbolC}{m}{n}
\DeclareFontShape{U}{MnSymbolC}{m}{n}{%
    <-6>  MnSymbolC5
   <6-7>  MnSymbolC6
   <7-8>  MnSymbolC7
   <8-9>  MnSymbolC8
   <9-10> MnSymbolC9
  <10-12> MnSymbolC10
  <12->   MnSymbolC12}{}
\DeclareMathSymbol{\powerset}{\mathord}{MnSyC}{180}

\makeatletter
\def\namedlabel#1#2{\begingroup
    #2%
    \def\@currentlabel{#2}%
    \phantomsection\label{#1}\endgroup
}
\makeatother

\numberwithin{theorem}{section}
\setlist[itemize]{leftmargin=1cm}
\setlist[enumerate]{leftmargin=1cm}

\everydisplay{\thickmuskip=7mu}
\renewcommand{\leq}{\leqslant}
\renewcommand{\geq}{\geqslant}

\renewcommand{\to}{\rightarrow}

\let\epsilon\varepsilon%

\def\cD{\mathcal{D}}

\def\cG{\mathcal{G}}
\def\cH{\mathcal{H}}

\def\cP{\mathcal{P}}

\def\PP{\mathbb{P}}

\def\1{\mathbbm{1}}
\def\<{\langle}
\def\>{\rangle}

\let\theta=\vartheta%
\let\rho=\varrho%
\let\phi=\varphi%

\DeclareMathOperator{\smp}{smp}

\DeclareMathOperator{\Sh}{Sh}

\usepackage{cases}
\usepackage{eqparbox}

\usepackage{lineno}

\begin{document}
\onehalfspace%
\date{\today, \currenttime}
\footskip=28pt

\title{On the Ramsey numbers of daisies I}

\author{Pavel Pudl\'{a}k}

\author{Vojtech R\"{o}dl}

\author{Marcelo Sales}

\thanks{The first author was supported by the project EPAC, funded by the Grant Agency of the Czech Republic under the grant agreement no. 19-27871X, and the institute grant RVO: 67985840. The second and third author were supported by NSF grant DMS 1764385. The second author was also supported by NSF grant DMS 2300347 and the third author by US Air Force grant FA9550-23-1-0298.}

\address{Institute of Mathematics, CAS, 
    Prague, Czech Republic}
\email{pudlak@math.cas.cz}

\address{Department of Mathematics, Emory University, 
    Atlanta, GA, USA}
\email{vrodl@emory.edu}

\address{Department of Mathematics, University of California, 
    Irvine, CA, USA}
\email{mtsales@uci.edu}

\begin{abstract}
Daisies are a special type of hypergraph introduced by Bollob\'{a}s,
Leader and Malvenuto in~\cite{BLM}. An $r$-daisy determined by a pair of disjoint sets
$K$ and $M$ is the $(r+|K|)$-uniform hypergraph $\{K\cup P:\: P\in M^{(r)}\}$. 
The authors of~\cite{BLM} studied Tur\'{a}n type
density problems for daisies. This paper deals with Ramsey numbers of
Daisies, which are natural generalizations of classical Ramsey
numbers. We discuss upper and lower bounds for the Ramsey number of $r$-daisies and also for special cases where the size of the kernel is bounded.
\end{abstract}

\maketitle

\section{introduction}\label{sec:intro}

All hypergraphs and set systems will be defined on $[n]$, where $[n]$ denotes the set $\{1, 2,\ldots ,n\}$. For a set $X$, we denote by $X^{(r)}$ the set of all subsets of $X$ of size $r$. When needed, we will use the natural ordering of $[n]$.

\begin{definition}\label{def:daisies}
For a pair of disjoint sets $K$ and $M$ and an integer $r\geq 2$ the \emph{daisy} $\cD_r(K,M)$ is the hypergraph defined as follows
\[
\cD_r(K,M)=\{K\cup P:\: P\in M^{(r)}\},
\]
i.e., an edge $X$ of $\cD_r(K,M)$ is the union of the \emph{kernel} $K$ and a \emph{petal} $P$. The set $M$ is called the \emph{universe of petals} of $\cD_r(K,M)$. When the pair $K, M$ have sizes $|K|=k$ and $|M|=m$ we say that the $(r+k)$-graph $\cD_r(K,M)$ is a \emph{$(r,m,k)$-daisy}. When the size of the kernel is not specified, we talk about an \emph{$(r,m)$-daisy} 
\end{definition}

Daisies were introduced by Bollob\'{a}s, Leader and Malvenuto
in~\cite{BLM}. They asked how many edges an $(r+k)$-graph can have
without containing an $(r,m,k)$-daisy. These problems are open
except for a few special cases. For instance, it is not known how
large the density of a $3$-graph without a $(2,4,1)$-daisy can be. In a recent breakthrough, Ellis, King, Ivan and Leader \cites{EIL24,EK23} proved lower bounds for some families of daisies. In particular, their result disproves a conjecture in \cite{BLM} that the Tur\'{a}n density of an $(r,m,k)$-daisy approaches $0$ for fixed $r, m$ and $k \rightarrow \infty$. In this paper we will be interested in Ramsey numbers of daisies. While in \cite{BLM} the main interest of the authors was in daisies with a kernel of fixed size, we will focus on daisies with unbounded kernel.

\begin{definition}
For integers $\ell\geq 2$ and $m\geq r \geq 2$, we define 
the \emph{Ramsey number of daisies} $D_r(m;\ell)$ as the minimum integer $n$ with the property that any coloring $\phi: \cP([n])\rightarrow [\ell]$ of the subsets of
$[n]$ by $\ell$ colors yields a monochromatic copy of an
$(r,m)$-daisy.\footnote{The colors of sets smaller than $r$ play no
role, but it will be simpler to talk about colorings of all sets.}
\end{definition}
Note that in our definition of $D_r$ we do not specify the size of the kernel of
the monochromatic daisy. In Section~\ref{sec:fixed},
we will also consider \emph{Ramsey numbers $D_r(m,k;\ell)$ of daisies with kernel of fixed size $k$} defined analogously.

Our motivation for studying daisies comes from theoretical computer science, specifically, a problem in the area of randomness extractors. In~\cite{CS15} Cohen and Shinkar studied the concept of the bit-fixing extractor (more precisely, extractor for bit-fixing sources). Defining this concept exactly would take us too far afield and is not the focus of the paper. However, this area of research is connected with Ramsey theory (e.g, \cite{PR}) and moreover with daisies (see Remark \ref{rmk:superdaisy}). To estimate the Ramsey number of daisies seems to be
a simpler problem that could shed light on the problems about
bit-fixing extractors. For the reader interested in randomness extractors, we recommend Shaltiel's survey \cite{S11}.


The problem of getting better estimates for the Ramsey number of
daisies could be of independent interest as well, since these numbers can be viewed as generalizations of Ramsey numbers for hypergraphs and there is a huge gap between currently available upper and lower bounds. The aim of this paper is to popularize this question. To this end we will prove some simple results that show
connections between Ramsey numbers of daisies and the standard Ramsey
numbers of complete hypergraphs. 


The paper is organized as follows. In the following section we
present our results and some definitions. In Section \ref{sec:randomlb} we prove Proposition \ref{prop:randomlb}, Section \ref{sec:ramseyvsdaisy} deals with Theorem \ref{prop:ramseyvsdaisy}, while in Section \ref{sec:superdaisy} we give a proof of Proposition \ref{prop:superdaisy}. In Section \ref{sec:fixed} we discuss some related problems for daisies of fixed kernel size or specified kernel position with respect to the underlying order of $[n]$. Finally, Section \ref{sec:remarks} is devoted to open problems and final remarks. 

\section{Definitions and results}\label{sec:def}


\subsection{Daisies of unrestricted kernel}

For integers $\ell\geq 2$ and $m\geq r\geq 2$, let $R_r(m;\ell)$ be the standard Ramsey number for $r$-graphs, i.e., the minimum number $n$ such that for every coloring of $[n]^{(r)}$ by $\ell$ colors, there exists a monochromatic complete $r$-graph on $m$ vertices $K^{(r)}_m$. Since the complete $r$-graph $K_m^{(r)}$ is an $(r,m)$-daisy with
empty kernel, we have the basic bound
\begin{align}\label{eq:eq1}
D_r(m;\ell)\leq R_r(m;\ell).
\end{align}

Setting $R_r(m)=R_r(m;2)$ and $D_r(m)=D_r(m;2)$ we first recall that Erd\H{o}s, Hajnal and Rado (see \cites{EHR65, GRS}) proved that there exist positive  $c_1:=c_1(r)$ and absolute positive constant $c_2$ such that
\begin{align}\label{eq:eq2}
    t_{r-2}(c_1m^2)\leq R_r(m)\leq t_{r-1}(c_2 m)
\end{align}
where the tower function $t_i$ is defined recursively as $t_0(x)=x$
and $t_{i+1}=2^{t_i(x)}$. It is conceivable that a similar bound holds true for $D_r$, but the standard techniques, in particular the
Stepping-up Lemma and shift graphs, do not seem to work for daisies. The probabilistic argument does work, but it only gives a lower bound that is far from the upper bound.




\begin{proposition}\label{prop:randomlb}
For integers $m,\ell,r\geq 2$, there exists positive constant $c:=c_r$ such that
\begin{align}\label{eq:wide}
D_r(m;\ell) \geq \ell^{cm^{r-1}}
\end{align}
holds.
\end{proposition}
The proof is a standard application of the Lov\'asz Local Lemma. For the upper bound, we were unable to say more than the bound on (\ref{eq:eq1}) for general values of $m$. However, for small values of $m$ we show that there exists a tower gap between $D_r(m)$ and $R_r(m)$. 



\begin{theorem}\label{prop:ramseyvsdaisy}
Given $\epsilon>0$, there exists $r_0$ such that for every $r\geq r_0$ the following holds. If $(1+\epsilon) r\leq m\leq (2-\epsilon)r$, then
\begin{align*}
    t_{\epsilon r/2}(D_r(m))\leq R_r(m).
\end{align*}
\end{theorem}

The proof of Theorem \ref{prop:ramseyvsdaisy} relies on an observation that the complement of a daisy is still a daisy and on a result on tower-type lower bounds on $R_r(m)$ when $m<2r$. For $m>2r$ we were unable to rule out even the possibility that $D_r(m;\ell)=R_r(m;\ell)$.

\begin{conj}
Let $m,r,\ell\geq 2$ be integers, then $\lim_{m\rightarrow \infty}\frac{D_r(m;\ell)}{R_r(m;\ell)}=0$.
\end{conj}

While there is a wide gap between the bounds given in (\ref{eq:eq1}) and (\ref{eq:wide}), with a slightly stronger definition of a daisy one can show fairly tight bounds. 



\begin{definition}
An $(r,m)$-\emph{superdaisy} $\cD_{\leq r}(K,M)$ with kernel $K$ and universe of petals $M$ of size $|M|=m$ is the system of all sets $X$ with
\begin{enumerate}
\item $K\subseteq X\subseteq K\cup M$,
\item $|X|\leq |K|+r$.
\end{enumerate}
\end{definition}
In other words, $\cD_{\leq r}(K,M)=\{K\}\cup \bigcup_{i=1}^r
\cD_i(K,M)$ where $\cD_i(K,M)$ is the daisy of kernel $K$ and universe
of petals $M$ with petal of size $i$.  It is easy to come up with a coloring that avoids monochromatic superdaisies. Indeed, one can take for example the coloring $\phi: \cP([n])\to \{0,1\}$ given by $\phi(X)=|X| \pmod{2}$. For this reason, instead of requiring the entire superdaisy to be monochromatic, we will be interested only when each daisy in the superdaisy is monochromatic.
We say that a superdaisy $\cD_{\leq r}(K,M)$ is \emph{level homogeneous,} if $\cD_i(K,M)$ is monochromatic for each $1\leq i \leq r$. We will denote by $D_{\leq r}(m;\ell)$ {the minimum $n$ such that for every coloring of all subsets of $[n]$ by $\ell$ colors, there exists a level homogeneous $(r,m)$-superdaisy.}


\begin{theorem}\label{prop:superdaisy}
If $\ell, r\geq 2$, then
\begin{align*}
R_r(m;\ell)\leq D_{\leq r}(m;\ell)\leq R_r(m+r-1;\ell^{r})
\end{align*}
holds for every $m>r$.
\end{theorem}

\begin{remark}\label{rmk:superdaisy}
The concept of superdaisies is closely related with bit-fixing extractors. Indeed, a bit-fixing extractor for sources of size $r$ and entropy $e$ is,
essentially, a coloring $\chi:2^{[n]}\rightarrow [2^e]$ of the subsets of $[n]$ by $2^e$ colors such that for every $(r,r)$-superdaisy $\cD$ all~$2^e$
colors are represented in~$\cD$ with almost the same
frequency.
\end{remark}



\subsection{Daisies of fixed kernel}

Another variant is the study of Ramsey number of daisies of bounded kernel. Analogously as in $D_r(m;\ell)$ we define the \emph{Ramsey number $D_r(m,k;\ell)$ of daises with kernel of size $k$} as the minimum integer $n$ with the property that any coloring $\phi:[n]^{(k+r)}\rightarrow [\ell]$ of the $(k+r)$-tuples of $[n]$ by $\ell$ colors yields a monochromatic copy of an $(r,m,k)$-daisy. The next theorem shows that for daisies of restricted kernel size, one can obtain tower-type lower bounds.

\begin{theorem}\label{thm:fixedkernel}
Let $k,\ell, r\geq 2$ be integers. Then the following two statements hold:
\begin{enumerate}
\item[(i)] $D_r(m,k;\ell^r)\geq R_r(\lceil m/(k+1) \rceil;\ell)$.
\item[(ii)] $D_r(m,k;\ell)\geq R_{r-k}(\lceil m/(k+1) \rceil - k;\ell)$ for $r>k$.
\end{enumerate}
\end{theorem}

Although a very natural variant of the problem, most of our paper will focus on the unrestricted version of Daisies. A follow-up paper~\cite{MS} will study the restricted variant in more detail.



\section{A lower bound on $D_r(m;\ell)$}\label{sec:randomlb}


The proof of Proposition \ref{prop:randomlb} is a standard application of the well-known Lov\'asz Local Lemma. The lemma was introduced by Erd\H{o}s and Lov\'{a}sz in \cite{EL}. 

\begin{lemma}[\cite{AS16}, Corollary 5.1.2]\label{lem:lll}
Let $A_1$, $A_2$, ..., $A_n$ be events in an arbitrary probability space. Suppose that each event $A_i$ is mutually independent of a set of all the other events $A_j$ but at most $d$, and that $\PP(A_i)\leq p$ for all $1\leq i \leq n$. If
\begin{align*}
    ep(d+1)\leq 1,
\end{align*}
then $\PP\left(\bigwedge_{i=1}^n\overline{A_i}\right)>0$
\end{lemma}

Proposition $\ref{prop:randomlb}$ follows from the next result by setting $n=\ell^{cm^{r-1}}$ for sufficiently small~$c$.

\begin{proposition}\label{prop:computlb}
If $n,m,r,\ell\geq 2$ and $n\geq m \geq r$ are integers satisfying
\begin{align*}
\binom{m}{r}n^m\ell^{1-\binom{m}{r}}<1,
\end{align*}
then $D_r(m;\ell)\geq n$.
\end{proposition}

\begin{proof}
Consider a random $\ell$-coloring $\phi:\cP([n])\rightarrow [\ell]$, where each set is colored uniformly and independently. Let $M, K\subseteq [n]$ with $|M|=m$ such that $M\cap K=\emptyset$. For a daisy $\cD(K,M)=\{K\cup P:\: P \in M^{(r)}\}$, let $A(K,M)$ be the event that $\cD(K,M)$ is monochromatic with respect to the coloring $\phi$. Clearly
\[
p:=\PP(A(K,M))=\ell^{1-\binom{m}{r}}.
\]

Note that the event $A(K,M)$ only depends on the events $A(K',M')$ such that the corresponding daisies $\cD(K,M)$ and $\cD(K',M')$ are not edge disjoint. We will estimate the degree $d$ of the dependency graph. That is, the number of daisies sharing an edge with $\cD(K,M)$. Let $X\in \cD(K,M)$ be an edge of our daisy. There are $\binom{k+r}{r}$ possibilities for an $r$-set $P'$ to be a petal determined by $X$ in some $\cD(K',M')$. Once the petal $P'$ is determined, then $K'=X\setminus P'$. It remains to determine the rest of the universe of petals $M'$, i.e., $m-r$ elements outside of $X$. There are $\binom{n-k-r}{m-r}$ possibilities for these elements. Furthermore, we have to multiply it by the number of edges $X$ in $\cD(K,M)$, which is $\binom{m}{r}$. Hence the degree is 
\begin{align*}
d\leq \binom{m}{r}\binom{k+r}{r}\binom{n-k-r}{m-r}<\binom{m}{r}n^{m}.
\end{align*}
We need to show $ep(d+1)\leq 1$.
However, a simple computation shows that
\begin{align*}
ep(d+1)<\binom{m}{r}n^{m}\ell^{1-\binom{m}{r}}<1.
\end{align*}
Consequently, by Lemma \ref{lem:lll}, we have $\PP(\bigwedge_{K,M}\overline{A(K,M)})>0$, which means that for some $\ell$-coloring $\phi$, no $(r,m)$-daisy is monochromatic.
\end{proof}

\section{Proof of Theorem \ref{prop:ramseyvsdaisy}}\label{sec:ramseyvsdaisy}

In this section we prove Theorem \ref{prop:ramseyvsdaisy}. The main result in this section is the following lower bound on $2$-color Ramsey numbers when $m<2r$.

\begin{theorem}\label{thm:shiftramsey}
For $r\geq 3$ and $m\geq r+7$,
$$R_r(m)\geq 4 t_{r-2}(\lfloor(m-r+1)/4\rfloor).$$
\end{theorem}

Let us remark that the best known lower bounds on Ramsey number are due to Conlon, Fox and Sudakov \cite{CFS13} by a clever application of the Stepping-up lemma. They proved that
\[
R_r(m)\geq t_{r-2}(cm^2)
\]
for a positive constant $c>0$ and $m\geq 3r$. On the other hand, Theorem \ref{thm:shiftramsey} can be applied for a wider range of $m$. The proof of Theorem \ref{thm:shiftramsey} will use a different approach, which is based on shift graphs. Similar results were established in \cites{DLR95, PR24}.

Let $\Sh(n,r)$ be the \emph{shift graph} on $[n]^{(r)}$ given by
\[
E(\Sh(n,r))=\{\{\{x_1,\ldots,x_r\},\{x_2,\ldots,x_{r+1}\}\}:\: x_1<x_2<\ldots<x_{r+1} \in [n]\}.
\]
For our exposition we will find more convenient to consider $\Sh(n,r)$ as a directed graph with directed pairs $(\{x_1,\ldots,x_r\},\{x_2,\ldots,x_{r+1}\})$. When talking about chromatic numbers of directed graphs we will mean the ``usual'' chromatic number of a symmetric graph which results by replacing each directed edge $(x,y)$ by the unordered pair $\{x,y\}$.

The proof of Theorem \ref{thm:shiftramsey} relies on the following proposition. Let $f(t,\ell)$ be the largest integer $m$ with the property that for any directed graph $D$ with $\chi(D)=m$, there exists an $\ell$-coloring of its arcs with no monochromatic path with $t$ vertices. The next proposition was proved in \cite{DLR95}. Since the proof is simple we will sketch it here.

\begin{proposition}\label{prop:fact2}
$f(t,\ell)\geq (t-1)^{\lfloor \ell/2\rfloor}$ for $\ell\geq 2$ and $t\geq 3$.
\end{proposition}

\begin{proof}[Sketch of proof of Proposition \ref{prop:fact2}]
Consider a digraph $D$ with $\chi(D)=(t-1)^{\lfloor \ell/2 \rfloor}$. Let $V(D)=\bigcup_{a\in [t-1]^{\lfloor \ell/2 \rfloor}} V_{a} $ be a color class partition of the vertices of $D$ indexed by the lattice points of the cube $[t-1]^{\lfloor \ell/2 \rfloor}$. We will color arcs of $D$ with $2\lfloor \ell/2 \rfloor \leq \ell$ colors as follows. For two color classes $V_{a}$ and $V_{b}$, let $j$ be the smallest index such that $a_j\neq b_j$. If $a_j<b_j$ color all the arcs $(x,y)\in V_{a}\times V_{b}$ by $2j-1$. On the other hand, if $a_j>b_j$, then color the arcs by $2j$. Clearly each directed path in each color has at most $t-1$ vertices.
\end{proof}

For a directed graph $G$, let $\partial G$ be the directed graph defined by $V(\partial G)=A(G)$ and 
\begin{align*}
A(\partial G)=\left\{((x,y),(z,w))\in A(G)^{(2)}:\: y=z\right\}. 
\end{align*}
That is, the vertices of $\partial G$ are the arcs of $G$ and the arcs of $\partial G$ are the oriented paths of length $2$. The next proposition is a corollary from Proposition \ref{prop:fact2} when $t=3$. 

\begin{proposition}[Lemma 5.2, \cite{PR24}]\label{prop:fact1}
Let $s\geq 1$ be an integer and $G$ a directed graph. If $\chi(\partial G)>2s$, then $\chi(G)>2^s$.
\end{proposition}

By considering $\Sh(n,r)$ as a directed graph, note that $\partial Sh(n,r)$ can be seen as a copy of $\Sh(n,r+1)$. Indeed, we can map the vertex $\{x_1,\ldots,x_{r+1}\}\in V(\Sh(n,r+1))$ with $x_1<\ldots<x_{r+1}$ to the directed pair $(\{x_1,\ldots,x_r\},\{x_2,\ldots,x_{r+1}\})$ in $V(\partial \Sh(n,r))$. The following is an immediate corollary.

\begin{corollary}\label{cor:chromaticshift}
Let $n,r,s\geq 1$ be integers. If $\chi(\Sh(n,r))>2s$, then $\chi(\Sh(n,r-1))>2^s$.
\end{corollary}

Now we are ready to give a proof of Theorem \ref{thm:shiftramsey}.

\begin{proof}[Proof of Theorem \ref{thm:shiftramsey}]
Set $n=R_r(m)$. Our aim is to find a lower bound on $n$. Consider a $2$-coloring $\phi:[n]^{(r)} \to [2]$ of the $r$-tuples of $[n]$. Note that by our observation that $\partial \Sh(n,r-1)$ is a copy of $\Sh(n,r)$, the coloring $\phi$ of the vertices of $\Sh(n,r)$ can be also interpreted as a coloring of the arcs of $\Sh(n,r-1)$.

By the definition of $n$, there exists a monochromatic set $X$ of size $m$. Let $x_1<\ldots<x_m$ be the elements of $X$. Since all $r$-tuples $\{x_{i_1},\ldots,x_{i_r}\}$ are of the same color, we have in particular that the set of all edges of the directed path
\begin{align*}
\{x_1,\ldots,x_{r-1}\}, \{x_2,\ldots,x_r\}, \ldots, \{x_{m-r+2},\ldots,x_m\}
\end{align*}
in $\Sh(n,r-1)$ are monochromatic. Since $\phi$ is arbitrary, we obtain that $\Sh(n,r-1)$ is a directed graph such that any $2$-coloring of its arcs yields an oriented monochromatic path with $m-r+2$ vertices of $\Sh(n,r-1)$. Thus, Proposition \ref{prop:fact2} with $\ell=2$ and $t=m-r+2$ gives us that
\begin{align}\label{eq:shift}
\chi(\Sh(n,r-1))>(m-r+1).
\end{align}

By induction we will observe that $\chi(\Sh(n,r-i))>4t_{i-1}\left(\lfloor (m-r+1)/4\rfloor\right)$ for $1\leq i \leq r-1$. The base case $i=1$ is just inequality (\ref{eq:shift}). Suppose that $\chi(\Sh(n,r-i))>4t_{i-1}\left(\lfloor(m-r+1)/4\rfloor\right)$ for $i\geq 1$. Then by Corollary \ref{cor:chromaticshift}, we obtain that
\begin{align*}
\chi(\Sh(n,r-i-1))>2^{2t_{i-1}\left(\lfloor(m-r+1)/4\rfloor\right)}\geq 4t_i\left(\left\lfloor(m-r+1)/4\right\rfloor\right)
\end{align*}
if $\lfloor(m-r+1)/4\rfloor\geq 2$, which holds for $m\geq r+7$. This finishes the induction. The result now follows since for $i=r-2$, we have $n=\chi(\Sh(n,1))>4t_{r-2}\left(\lfloor(m-r+1)/4\rfloor\right)$.
\end{proof}

We finish the Section by proving the theorem stated in the introduction. 

\begin{proof}[Proof of Theorem \ref{prop:ramseyvsdaisy}]
First, we are going to prove that $D_{r}(m)=D_{m-r}(m)$ for integers $m>r\geq 1$. Recalling Definition \ref{def:daisies}, the equality $D_r(m)=D_{m-r}(m)$ means that the existence of a $2$-coloring $\phi$ of $\cP([n])$ without a monochromatic $(r,m)$-daisy can be used to find a coloring $\psi^*$ which yields no monochromatic $(m-r,m)$-daisy. Let $n=D_{r}(m)-1$ and let $\phi:\cP([n])\rightarrow [2]$ be a $2$-coloring of all the subsets of $[n]$ without a monochromatic $(r,m)$-daisy. Consider the complementary coloring $\phi^*:\cP([n])\rightarrow [2]$ given by
\begin{align*}
\phi^*(X)=\phi([n]\setminus X)
\end{align*}
for every $X\subseteq [n]$.

We claim that $\phi^*$ does not contain monochromatic $(m-r,r)$-daisies. Suppose to the contrary that $\cD^*$ is a monochromatic $(m-r,r)$-daisy under the coloring $\phi^*$. Then consider the dual graph $\cD$ given by
\begin{align*}
    \cD=\{[n]\setminus X:\: X\in \cD^*\}.
\end{align*}
The hypergraph $\cD$ is a monochromatic $(r,m)$-daisy under the coloring $\phi$, contradicting our choice of $\phi$. Hence, $D_{m-r}(m)\leq D_r(m)$. Since $1\leq r < m$ is arbitrary, it follows that $D_{r}(m)=D_{m-r}(m)$.

Note that Erd\H{o}s, Hajnal and Rado upper bound on Ramsey numbers in (\ref{eq:eq2}) holds for $m\geq 1$. Together with the observation in $(\ref{eq:eq1})$, we obtain that
\begin{align}\label{eq:updaisy}
D_r(m)=D_{m-r}(m)\leq R_{m-r}(m)\leq t_{m-r-1}(cm)\leq t_{m-r-1}(2cr)
\end{align}
for $m<2r$ and absolute positive constant $c$. Moreover, Theorem \ref{thm:shiftramsey} gives for $m\geq (1+\epsilon)r$ that
\begin{align}\label{eq:downramsey}
    R_r(m)\geq t_{r-2}(\lfloor(m-r+1)/4\rfloor)\geq t_{r-2}(\epsilon r/4)
\end{align}
Then the statement follows for sufficiently large $r$ by combining (\ref{eq:updaisy}) and (\ref{eq:downramsey}). Indeed,
\begin{align*}
    t_{\epsilon r/2}(D_r(m))\leq t_{m-r-1+\epsilon r/2}(2cr)\leq t_{r-2}(\epsilon r/4)\leq R_r(m),
\end{align*}
for $m<(2-\epsilon)r$ and sufficiently large $r$.
\end{proof}

\section{Proof of Theorem \ref{prop:superdaisy}}\label{sec:superdaisy}

We start this section by noting that since the complete $i$-graph $K_m^{(i)}$ is an $(i,m)$-daisy with empty kernel, one could restrict the problem to find a level homogeneous superdaisy to the problem of finding a level homogeneous daisy with empty kernel. In this case, the problem is the same as finding a set $M\subseteq [n]$ that is a monochromatic with respect to all possible uniformities from $1$ to $r$. It turns out that we are able to bound the last problem only using the Ramsey number for $r$-tuples.

\begin{proposition}\label{prop:upsuper}
\[
D_{\leq r}(m;\ell)\leq R_r(m+r-1;\ell^{r}) 
\]
\end{proposition}

\begin{proof}
Let $n=R_r(m+r-1;\ell^{r})$ and consider an arbitrary $\ell$-coloring $\phi: \cP([n])\to [\ell]$ of the subsets of $[n]$. We define an $\ell^{r}$-coloring $\psi: [n]^{(r)}\to [\ell]^{r}$ of the $r$-tuples of $[n]$ given as follows. Let $X=\{x_1,\ldots,x_r\} \in [n]^{(r)}$ be an $r$-tuple of $[n]$. Then 
\begin{align*}
    \psi(X)_i=\phi(\{x_1,x_2,\ldots,x_i\})
\end{align*}
for every $1\leq i \leq r$. By our choice of $n$, there is a set $Y \subseteq [n]$ of size $m+r-1$ monochromatic with respect to the coloring $\psi$. Let $M$ be the subset consisting of the first $m$ elements of $Y$. Thus, by the definition of $\psi$, for every $1\leq i \leq r$ the family $M^{(i)}$ is monochromatic. This implies that $\cD_{\leq r}(\emptyset, M)$ is level homogeneous.


\end{proof}

A more careful analysis of the standard proof of Ramsey's theorem (see e.g. \cite{GRS}) gives a similar upper bound as in Proposition \ref{prop:upsuper} (with even better constants). However, for the sake of simplicity we decided to present the proof above. 

For the lower bound, we were able to show that $D_{\leq r}(m;\ell)$ is at least the Ramsey number for $r$-graphs in $\ell$ colors.

\begin{proposition}
\[
D_{\leq r}(m;\ell)\geq R_r(m;\ell)
\]
\end{proposition}

\begin{proof}
Let $n=R_r(m;\ell)-1$ and consider an $\ell$-coloring $\phi:[n]^{(r)}\to \{0,1,\ldots,\ell-1\}$ of the $r$-tuples of $[n]$ with no monochromatic clique $K^{(r)}_m$. We define a coloring $\psi: \cP([n])\rightarrow \{0,1,\ldots,\ell-1\}$ by
\begin{align*}
\psi(X)=\begin{cases}
\sum_{R\in X^{(r)}}\phi(R) \pmod{\ell}, &\quad \text{if $|X|\geq r$}\\
0, &\quad \text{otherwise},
\end{cases}
\end{align*}
for $X\subseteq [n]$. We claim that $\psi$ is a coloring with no level homogeneous $(r,m)$-superdaisy.

Assume by contradiction that there exists a $(r,m)$-superdaisy $\cD_{\leq r}(K,M)$ that is level homogeneous with $|K|=k$ and $|M|=m$. Since $\cD_i(K,M)$ is monochromatic for each $1\leq i \leq r$, let $c_i$ be the color of the edges of size $k+i$.

Let $\psi_K$ be an auxiliary $\ell$-coloring $\psi_K:M^{(\leq r)} \to \{0,1,\ldots,\ell-1\}$ of the subsets $P\subseteq M$, $|P|\leq r$ defined by
\begin{align*}
\psi_K(P)=\begin{cases}
\sum_{J\in K^{(r-|P|)}}\phi(J\cup P) \pmod{\ell}, &\quad \text{if $|K\cup P|\geq r$}\\
0, &\quad \text{otherwise}.
\end{cases}
\end{align*}
In other words, $\psi_K(P)$ is the contribution of the coloring $\phi$ of the $r$-edges containing $P$ from $[n]^{(r)}$ to the color of $\psi(K\cup P)$. 

We will show by induction that there are constants $d_i\in\{0,1,\ldots,\ell-1\}$ for $0\leq i \leq r$ such that $\psi_K\mid_{M^{(i)}}\equiv d_i$, i.e., for every petal $P\in M^{(i)}$ we have $\psi_K(P)=d_i$.

For $i=0$ there is nothing to do, because $M^{(0)}=\{\emptyset\}$. In this case just take $d_0=\psi_K(\emptyset)$. Now suppose that $i>0$ and let $P\in M^{(i)}$. We claim that
\begin{align}\label{eq:restriction}
\psi_K(P)=\psi(K\cup P)-\sum_{L\subsetneq P} \psi_K(L).
\end{align}
If $|K\cup P|<r$, then equation (\ref{eq:restriction}) holds since $\psi(K\cup P)=0$ and $\psi_K(L)=0$ for every $L\subseteq P$. Now let $t=\min\{k,r\}$. Then
\begin{align*}
\psi(K\cup P)=\sum_{\ell=r-t}^i\sum_{J\in K^{(r-\ell)}}\sum_{L\in P^{(\ell)}}\phi(J\cup L)=\sum_{\ell=r-t}^i\sum_{L\in P^{(\ell)}}\psi_K(L)
\end{align*}
Since by definition any set $L\subseteq P$ of size $|L|<r-t$ is such that $\psi_K(L)=0$, we have 
\begin{align*}
    \psi(K\cup L)=\sum_{\ell=0}^i\sum_{L\in P^{(\ell)}}\psi_K(L)=\psi_K(P)+\sum_{L\subsetneq P}\psi_K(L),
\end{align*}
which proves (\ref{eq:restriction}).

Using that $\psi(K\cup P)=c_i$ and by the induction hypothesis, we obtain
\begin{align*}
\psi_K(P)=c_i-\sum_{j=0}^{i-1}\binom{i}{j}d_j,
\end{align*}
which does not depend on the choice of $P$. Thus setting $d_i=c_i-\sum_{j=0}^{i-1}\binom{i}{j}d_j$ gives us that $\psi_K\mid_{M^{(i)}}\equiv d_i$.

In particular, for $i=r$, the last paragraph shows that $M^{(r)}$ is monochromatic with respect to $\psi_K$. Since $\psi_K(P)=\phi(P)$ for every $P\in M^{(r)}$, we have that $M^{(r)}$ is a monochromatic $K_m^{(r)}$ with respect to the coloring $\phi$, which contradicts our assumption on $\phi$.
\end{proof}

\section{Simple daisies and daisies of kernel of given size}\label{sec:fixed}

We now study special cases of the daisy problem. We start by defining a simple daisy, a special type of daisy with the property that the universe of petals separates the kernel in two parts.

\begin{definition}
A daisy $\cD_r(K,M)$ is \emph{simple} if there exists a partition $K=K_0\cup K_1$ of the kernel such that $K_0<M<K_1$, i.e., $\max(K_0)<\min(M)\leq \max(M)<\min(K_1)$.
\end{definition}

Although we do not have good lower bounds for the Ramsey number of daisies, we can derive better lower bounds for simple daisies at the cost of increasing the number of colors. More precisely, let $D^{\smp}_r(m;\ell)$ denote the minimum integer $n$ such that any $\ell$-coloring of $\cP([n])$ yields a monochromatic simple $(r,m)$-daisy.

\begin{proposition}\label{prop:simple}
If $\ell,r \geq 2$, then
\begin{align*}
D^{\smp}_r(m;\ell^r)\geq R_r(m;\ell)
\end{align*}
\end{proposition}

\begin{proof}
Let $\phi:[n]^{(r)}\to \{0,1,\ldots,\ell-1\}$ be an $\ell$-coloring of the $r$-tuples of $[n]$ with no monochromatic $K_m^{(r)}$. We will define a coloring $\psi: \cP([n])\to \{0,1,\ldots, \ell-1\}^r$ from the subsets of $[n]$ to the ordered $r$-tuples of $\{0,1,\ldots,\ell-1\}$ as follows. Let $X=\{x_1,\ldots,x_t\}\subseteq [n]$. We define $\psi(X)$ by 
\begin{align*}
    \psi(X)_i=\sum_{s= i \pmod{r}}\phi(\{x_s,x_{s+1}\ldots,x_{s+r-1}\}) \pmod{r}
\end{align*}
for $1\leq i \leq r$. In other words, the $i$-th coordinate of $\psi(X)$ is given by summing the colors of all the blocks of $r$ consecutive elements of $X$ starting with an index congruent to $i \pmod{r}$.

Suppose that there is a monochromatic simple daisy $\cD=\cD(K_0\cup K_1,M)$ with respect to $\psi$ and assume that $|K_0|=i_0 \pmod{r}$. Let $P, P' \in M^{(r)}$ and consider the edges $X=K_0\cup P\cup K_1=\{x_1,\ldots,x_t\}$ and $X'=K_0\cup P'\cup K_1=\{x_1',\ldots,x_t'\}$ in $\cD$. Note that the petals $P$ and $P'$ corresponds to the block of $r$ consecutive elements starting with the index $i_0+1$ in the edges $X$ and $X'$, respectively. Moreover, because $|P|=|P'|=r$, all the other blocks of $r$ consecutive elements in $X$ and $X'$ starting with index congruent to $i_0+1 \pmod{r}$ are completely inside $K_0\cup K_1$ and consequently $\phi(\{x_s,\ldots,x_{s+r-1}\})=\phi(\{x_s',\ldots,x_{s+r-1}'\})$ for $s= i_0+1 \pmod{r}$ and $s\neq i_0+1$. Therefore, $\psi(X)_{i_0+1}=\psi(X')_{i_0+1}$ implies that $\phi(P)=\phi(P')$. Hence $M^{(r)}$ is monochromatic with respect to $\phi$, which contradicts our assumption on $\phi$. 
\end{proof}

A natural variation of the daisy problem introduced in Section \ref{sec:intro} is to determine the Ramsey number of daisies of fixed kernel size. To be more precise, we define $D_r(m,k;\ell)$ as the minimum integer $n$ with the property that any $\ell$-coloring $\phi:[n]^{(k+r)} \to [\ell]$ of the $(k+r)$-tuples of $[n]$ yields a monochromatic copy of an $(r,m,k)$-daisy. 

The first remark about the problem is that one can immediately obtain an upper bound from the original Ramsey number. Indeed, given an $\ell$-coloring $\phi:[n]^{(k+r)}\to [\ell]$ we can define a coloring $\phi':[n-k]^{(r)}\to [\ell]$ by
\begin{align*}
    \phi'(\{x_1,\ldots,x_{r}\})=\phi(\{x_1,\ldots,x_r,n-k+1,\ldots,n\}) 
\end{align*}
for $\{x_1,\ldots,x_r\}\in [n-k]^{(r)}$. It is not difficult to check that a monochromatic clique with respect to $\phi'$ corresponds to a monochromatic daisy of kernel $K=\{n-k+1,\ldots,n\}$ with respect to $\phi$. Hence,
\begin{align*}
D_r(m,k;\ell)\leq R_r(m;\ell)+k.
\end{align*}

In the rest of this section, we will give a proof of Theorem \ref{thm:fixedkernel}. Our approach to provide lower bounds for $D_r(m,k;\ell)$ is based on the concept of simple daisies given above. The concept was further explored in \cite{MS} to prove Theorem \ref{th:fixdaisy} below. As in the definition of $D_r(m,k;\ell)$, let $D_r^{\smp}(m,k;\ell)$ denote the minimum integer $n$ such that any $\ell$-coloring of the $(k+r)$-tuples of $[n]$ yields a monochromatic simple $(r,m,k)$-daisy.

\begin{proposition}\label{prop:pigeonhole}
$D_r(m,k;\ell)\geq D_r^{\smp}(\left\lceil m/(k+1) \right\rceil,k;\ell)$
\end{proposition}

\begin{proof}
Suppose that a $2$-coloring of $[n]^{(k+r)}$ contains a monochromatic copy of a $(r,m,k)$-daisy $\cD=\cD(K,M)$. Write $K=\{x_1,\ldots,x_k\}$ with $x_1<x_2<\ldots<x_k$ and let $x_0=0$ and $x_{k+1}=n+1$. For $0\leq i \leq k$, let $M_i$ be the vertices of $M$ between $x_i$ and $x_{i+1}$. By the pigeonhole principle there exists index $j$ such that $|M_j|\geq \left\lceil|M|/(k+1)\right\rceil \geq \left\lceil m/(k+1)\right\rceil$. Therefore, the induced subgraph $\cD[K\cup M_j]$ is a monochromatic simple $(r,\left\lceil m/(k+1) \right\rceil,k)$-daisy and hence $D_r(m,k;\ell)\geq D_r^{\smp}(\left\lceil m/(k+1)\right\rceil,k;\ell)$. 
\end{proof}

Note that a coloring without a monochromatic simple $(r,m)$-daisy does not contain in particular a simple $(r,m,k)$-daisy. Hence, $D_r^{\smp}(m;\ell)\leq D_r^{\smp}(m,k;\ell)$. 
As a quick corollary, Proposition \ref{prop:simple} and \ref{prop:pigeonhole} together with the observation above give Part (i) of Theorem \ref{thm:fixedkernel}.

\begin{corollary}
If $k,\ell,r \geq 2$, then
\begin{align*}
    D_r(m,k;\ell^r)\geq R_r(\lceil m/(k+1) \rceil;\ell).
\end{align*}
\end{corollary}

We were also able to prove a lower bound without increasing the numbers of colors (part (ii) of Theorem \ref{thm:fixedkernel}), but at the expense of reducing other parameters.

\begin{proposition}\label{prop:fox}
If $\ell \geq 2$ and $r>k\geq 2$, then
\[
D_r(m,k;\ell)\geq R_{r-k}(\lceil m/(k+1) \rceil - k;\ell)
\]
\end{proposition}

\begin{proof}
We will prove that $D_r^{\smp}(m,k;\ell)\geq R_{r-k}(m-k;\ell)$. The inequality in the statement follows from Proposition \ref{prop:pigeonhole}. Let $n=R_{r-k}(m-k;\ell)-1$ and let $\phi: [n]^{(r-k)}\to [\ell]$ be an $\ell$-coloring of the $(r-k)$-tuples of $[n]$ with no monochromatic copy of a $K_{m-k}^{(r-k)}$. We define a coloring $\psi: [n]^{(k+r)}\to [\ell]$ by
\[
\psi(\{x_1,\ldots,x_{k+r}\})=\phi(\{x_{k+1},\ldots,x_r\}),
\]
for $\{x_1,\ldots,x_{k+r}\}\in [n]^{(k+r)}$.

Suppose by contradiction that there exists a monochromatic simple daisy $\cD=\cD(K_0\cup K_1, M)$ with respect to $\psi$. Write $|K_0|=k_0$, $|K_1|=k_1$ and $M=\{x_1,\ldots,x_m\}$. We claim that the set $M'=\{x_{k-k_0+1},\ldots,x_{m-k+k_1}\}$ is monochromatic with respect to $\phi$. Indeed, if $A, A' \in M'^{(k-r)}$, then the sets
\begin{align*}
X=K_0\cup\{x_1,\ldots,x_{k-k_0}\}\cup A\cup \{x_{m-k+k_1+1},\ldots,x_m\}\cup K_1
\end{align*}
and
\begin{align*}
X'=K_0\cup\{x_1,\ldots,x_{k-k_0}\}\cup A'\cup \{x_{m-k+k_1+1},\ldots,x_m\}\cup K_1
\end{align*}
are edges of $\cD$, because $\{x_1,\ldots,x_{k-k_0}\}\cup A\cup \{x_{m-k+k_1+1}$ and $\{x_1,\ldots,x_{k-k_0}\}\cup A'\cup \{x_{m-k+k_1+1},\ldots,x_m\}$ are sets of size $(k-k_0)+(k-r)+(k-k_1)=r$ in $M$. Since $\cD$ is monochromatic, we obtain that $\phi(A)=\psi(X)=\psi(X')=\phi(A')$ and consequently $M'$ is set of size $m-k$ monochromatic with respect to $\phi$, which is a contradiction.
\end{proof}

\section{Concluding remarks}\label{sec:remarks}

\subsection{Better bounds on the Ramsey number of daisies with unrestricted kernels.} Likely, the most interesting problem is to improve the bounds on $D_r(m;\ell)$. Proposition \ref{prop:randomlb} gives us only an exponential lower bound. On the other hand, we were unable to rule out even that $D_r(m;\ell)=R_r(m;\ell)$ for sufficiently large $m$. While we do not believe that this is the case, we also believe that our lower bound is far from being the best possible. Thus raising the following question.

\begin{problem}
Improve the lower bound on $D_r(m)$. In particular, is it true that for every $r\geq 2$, there is $m_0=m_0(r)$ and $s=s(r)$ such that $D_r(m)$ grows like a tower of exponentials $t_s(m)$, where $s\to \infty$ as $r\to \infty$ and $m\geq m_0$?
\end{problem}

The following weaker question is already interesting.

\begin{problem}
    Is $D_r(m)>2^{2^{\epsilon m}}$ for some $r$ and $\epsilon>0$?
\end{problem}

\subsection{Lower bounds for daisies with fixed kernel size}

In Section \ref{sec:fixed} we considered the problem of determining $D_r(m,k;\ell)$, the Ramsey number of daisies of kernel of size $k$. We proved that it is lower bounded by $R_r(\lfloor m/(k+1) \rceil; p)$ if $\ell\geq p^r$. Clearly, this bound works only if $m$ is sufficiently large
with respect to $k$ and $r$. In the forthcoming paper~\cite{MS}
the junior author shows a bound for 2 colors. His proof uses a more involved
modification of the Stepping-up Lemma. 


\begin{theorem}[\cite{MS}]\label{th:fixdaisy}
Given integers $r \geq 3$ and $k\geq 0$, there exists a positive absolute constant $c>0$ and an integer $m_0=m_0(r,k)$ such that
\[
D_r(m,k;2)\geq t_{r-2}(ck^{-3}m^{1/2^{r-4}})
\]
holds for every $m\geq m_0$.
\end{theorem}

The bound is essentially of the same type as known bounds for
$R_r(m)$, but also in this theorem the size of the kernel $k$ must be
substantially smaller than $m$. So the natural next step toward proving
better bounds on the Ramsey numbers of daisies with unrestricted
kernel size is to prove a similar lower bound with $k>m$.

\subsection{Estimates for monochromatic directed paths}

Finally, we state a question arising in relation with the function $f(t,\ell)$ introduced in Section \ref{sec:ramseyvsdaisy}. Recall that we define $f(t,\ell)$ as the largest integer $m$ with the property that for any directed graph $D$ with $\chi(D)=m$, there exists an $\ell$-coloring of its arcs with no monochromatic directed path of size $t$. As shown, the function $f(t,\ell)$ can be used to determine lower bounds for Ramsey numbers. This leads to the natural problem of determining $f(t,\ell)$.

Proposition \ref{prop:fact2} shows that $f(t,\ell)\geq (t-1)^{\lfloor\ell/2\rfloor}$. For $\ell=2$, it is not hard to check that the result is actually tight and $f(t,2)=t-1$. However, the lower bound provided by the proposition does not seem optimal. Indeed, one can prove by a blow-up iterated construction that $f(t,\ell)=\Omega(t^{\log_2 \ell})$. Since $\log_2 3\approx1.584..>1$, this construction gives better bounds for $\ell=3$, but fails to give good bounds for large values of $\ell$. It would be interesting to know if it is possible to improve the bounds on $f(t,\ell)$ for odd $\ell$ in general.

\subsection{Ramsey number of $\cG*\cH$}

In \cite{BLM}, the authors proposed the following operation with hypergraphs. Given an $r$-graph $\cG$ and an $s$-graph $\cH$, we construct the $(r+s)$-graph $\cG*\cH$ with vertex set the disjoint union of the vertex sets of $\cG$ and $\cH$ and whose edges are all sets of the form $X\cup Y$ with $X\in \cG$ and $Y\in \cH$. Observe that if $\cG=K_k^{(k)}$ (a $k$-edge) and $\cH=K_m^{(r)}$, then $\cG*\cH$ is just an $(r,m,k)$-daisy. Similarly to \cite{BLM} one can ask the very general question. Given an $r$-graph $\cG$, let $R(\cG;\ell)$ be the minimum number $n$ such that for every coloring of $[n]^{(v(\cG))}$ by $\ell$ colors, there exists a monochromatic copy of $\cG$.

\begin{problem}
Let $\cG$ be an $r$-graph and $\cH$ be an $s$-graph. How does $R(\cG*\cH;\ell)$ compare to $R(\cG;\ell)$ and $R(\cH;\ell)$?
\end{problem}

Perhaps a more concrete question could be the following. Let $\cG^t$ denote the product $G*\ldots*G$ $t$ times.

\begin{problem}
What can one say about the growth of $R((K_3^{(2)})^t;\ell)$ as a function of $t$ and $\ell$? 
\end{problem}

\section*{Acknowledgments}
The authors thank Jacob Fox for fruitful discussions and the anonymous referee for helpful comments on the paper.

\bibliography{literature}

\end{document}